\documentclass[12pt]{amsart}
\usepackage{float} 
\usepackage{amsmath,amsthm,amssymb,verbatim,stmaryrd,xcolor,microtype,graphicx,aliascnt,needspace,mathtools,booktabs}
\usepackage[shortlabels]{enumitem}
\usepackage[T1]{fontenc}
\usepackage[utf8]{inputenc}
\usepackage[english]{babel} 
\usepackage[top=3.5cm,bottom=3.55cm,left=3.5cm,right=3.5cm]{geometry}
\usepackage[bookmarksopen=true,bookmarksopenlevel=2,bookmarksdepth=2,linktoc=page,colorlinks,linkcolor={red!70!black},citecolor={red!80!black},urlcolor={blue!80!black},pdfpagemode=UseOutlines,pdfstartview={fitH}]{hyperref}
\usepackage{longtable}
\usepackage{bm}
\usepackage{tikz}
\usepackage{tikz-cd}
\usepackage{euscript}                             
\usepackage{bbold}

\usepackage[colorinlistoftodos,bordercolor=orange,backgroundcolor=orange!20,linecolor=orange,textsize=footnotesize]{todonotes} \makeatletter \providecommand \@dotsep{5} \def\listtodoname{List of Todos} \def\listoftodos{\@starttoc{tdo}\listtodoname} \makeatother 

\usepackage{newtxtext,newtxmath}                              

\allowdisplaybreaks 

\usepackage{etoolbox}
\makeatletter
\patchcmd{\@startsection}{\@afterindenttrue}{\@afterindentfalse}{}{}             
\patchcmd{\part}{\bfseries}{\bfseries\LARGE}{}{}
\patchcmd{\section}{\scshape}{\bfseries}{}{}\renewcommand{\@secnumfont}{\bfseries} 
\patchcmd{\@settitle}{\uppercasenonmath\@title}{\large}{}{}
\patchcmd{\@setauthors}{\MakeUppercase}{}{}{}
\addto{\captionsenglish}{} 
\addto{\captionsenglish}{} 
\addto{\captionsenglish}{} 
\makeatother

\usepackage{fancyhdr}

\addto\extrasenglish{}  
\theoremstyle{plain}
\newtheorem{thm}{Theorem}[section] 
\newaliascnt{lemma}{thm}\aliascntresetthe{lemma}
\newaliascnt{cor}{thm}\newtheorem{cor}[cor]{Corollary}\aliascntresetthe{cor}

\newaliascnt{claim}{thm}\aliascntresetthe{claim}
\newtheorem*{claim*}{Claim}
\newtheorem*{thm*}{Theorem}
\newtheorem*{lem*}{Lemma}
\newtheorem*{cor*}{Corollary}

\theoremstyle{definition}
\newaliascnt{df}{thm}\newtheorem{df}[df]{Definition}\aliascntresetthe{df}
\newaliascnt{rem}{thm}\newtheorem{rem}[rem]{Remark}\aliascntresetthe{rem}
\newaliascnt{ex}{thm}\newtheorem{ex}[ex]{Example}\aliascntresetthe{ex}
\newaliascnt{conj}{thm}\aliascntresetthe{conj}
\newaliascnt{problem}{thm}\aliascntresetthe{problem}

\newtheorem*{df*}{Definition}
\newtheorem*{ex*}{Example}
\newtheorem*{rem*}{Remark}

\DeclareRobustCommand{\gobblefour}[5]{}    

\DeclareSymbolFont{sfoperators}{OT1}{bch}{m}{n} \DeclareSymbolFontAlphabet{\mathsf}{sfoperators} \makeatletter\def\operator@font{\mathgroup\symsfoperators}\makeatother 

\DeclareMathOperator{\Hom}{Hom}

\DeclareMathOperator{\Spec}{Spec}
\DeclareMathOperator{\BAff}{BAff}

\DeclareMathOperator{\colim}{colim\,}

\DeclareMathOperator{\Rings}{{Rings}}

\DeclareMathOperator{\Bands}{{Bands}}

\DeclareMathOperator{\Bend}{{Bend}}
\DeclareMathOperator{\bend}{{bend}}
\newcommand{\A}{{\mathbb A}}

\newcommand{\R}{{\mathbb R}}

\newcommand{\T}{{\mathbb T}}

\newcommand{\cO}{{\mathcal O}}

\renewcommand{\phi}{{\varphi}}

\renewcommand{\max}{\textup{max}}

\renewcommand{\geq}{\geqslant}
\renewcommand{\leq}{\leqslant}

\renewcommand{\setminus}{\backslash}

\renewcommand{\bm}{}

\DeclareMathOperator{\Trop}{Trop}

\DeclareMathOperator{\sgn}{sgn}

\usepackage{adjustbox}
\usepackage{caption}
\usepackage[
  backend=bibtex,
  style=alphabetic,
  doi=false,
  url=false,
  eprint=true,
  giveninits=true,
  maxbibnames=99,
  maxcitenames=99,
  maxalphanames=99,
  minalphanames=99
]{biblatex}
\title{On Bands and Limit Theorems in Tropical Geometry\\[10pt]}

\author{Arne Kuhrs}
\address{\rm Arne Kuhrs, University of Paderborn, Germany}
\email{kuhrs@math.uni-paderborn.de}

\author{Alejandro Mart\'inez M\'endez}
\address{\rm Alejandro Mart\'inez M\'endez, University of Groningen, the Netherlands}
\email{a.m.m.martinez.mendez@rug.nl}

\author{Pedro Souza}
\address{\rm Pedro Souza, Goethe-Universität Frankfurt am Main, Germany}
\email{souza@math.uni-frankfurt.de}

\hypersetup{pdftitle={On bands and Limit Theorems},pdfauthor={Arne Kuhrs, Alejandro Mart\'inez M\'endez and Pedro Souza}}

\begin{document}

\maketitle

\begin{abstract}

We review the basic theory of bands and band schemes introduced by Baker-Jin-Lorscheid, which is an algebraic framework for tropicalization, analytification, and $\mathbb{F}_1$-geometry. For an affine scheme $X$ over a non-Archimedean valued field $k$, one can associate to every affine embedding $\iota$ of $X$ a naturally defined affine band scheme $Y_\iota$ whose rational points over the tropical band $\T$ recover the tropicalization $\Trop(X,\iota)$. We prove that $X$ is the limit of the $Y_\iota$ in the category of band schemes, thereby obtaining a scheme-theoretic enhancement of Payne's limit theorem. By taking $\T$-rational points, this recovers Payne's theorem for affine tropicalizations from the perspective of band scheme theory and the same method provides an analogous result in the real tropical setting.
\end{abstract}

\tableofcontents

\section{Introduction}
\subsection*{Bands} 
A \emph{band} is an algebraic object that generalizes a commutative ring, except that it does not have an addition operation. Instead, it consists of a (pointed) multiplicative monoid together with a separate object called the \textit{null set} of the band, which captures its additive structure. Bands generalize not only rings, but also hyperrings, partial fields and other ``exotic'' algebraic structures. They were first introduced in \cite{https://doi.org/10.1112/jlms.70125}, with the goal of providing new algebraic foundations to $\mathbb{F}_{1}$-geometry and tropical geometry. The category of bands is large enough to admit fully faithful embeddings from various categories of algebraic objects, like the ones mentioned above, and is a well-behaved category in many ways (for example, it is complete and cocomplete). Bands build the algebraic objects that give rise to a corresponding geometric theory of band schemes. 
\newline




The notion of a \textit{tract} (introduced in \cite{BakerBowler2019Matroids}, see \ref{rem: Tracts}) is a direct precursor of the notion of a band and was developed to study various types of matroids (usual, oriented, valuated, etc) in a unified framework. Tracts and related algebraic structures also appear naturally in tropical geometry, where tropical varieties may be viewed as ``algebraic varieties over $\mathbb{T}$'' where $\T$ is a realization of the tropical semiring into an adequate algebraic setting. From the perspective of \cite{Giansiracusa2022UniversalTropicalization}, the object $\mathbb{T}$ is considered as a semiring where tropical varieties are defined via the bend relations (see also \cite[Lecture~2]{giansiracusa2026lecturestropicalalgebra}). In the language of bands, it is precisely the \emph{tropical band} (Definition~\ref{thm:exampletropical}) \cite[Ex.1.2.3]{https://doi.org/10.1112/jlms.70125}. These descriptions are not wholly independent, the relationship is detailed in \ref{rem: Bend relations}. This point of view will reappear in Section~\ref{Section: Limitsoftrops}, where the band schemes $Y_\iota$ attached to embeddings are constructed.


\subsection*{Tropicalization} 
Let $X$ be an affine variety over a non-Archimedean valued field $k$. 
The \textit{tropicalization} of $X$ is a polyhedral complex $\operatorname{Trop}(X,\iota)\subseteq\mathbb{R}^{n}$ whose definition depends on a choice of closed embedding $\iota \colon X \hookrightarrow \mathbb{A}^{n}_k$ of $X$ into some affine space $\mathbb{A}^{n}_{k}$. The point of this construction is to produce a combinatorial object which is simpler than the original variety $X$ but still carries enough information to study various questions about $X$. There are many cases where this strategy has produced positive results and we refer the interested reader to \cite[Chapter 1]{MaclaganSturmfels2015} for many applications of tropical methods in algebraic geometry and other fields.
\newline

A basic feature of tropicalization is that it depends on the choice of embedding $\iota$. Since tropicalization is functorial with respect to equivariant morphisms, it is natural to consider the inverse limit over all affine embeddings. In the affine case, Payne's theorem \cite{P09} identifies this limit with the Berkovich analytification of $X$:
\[
X^{\operatorname{an}} \;\cong\; \lim_{\iota} \Trop(X,\iota).
\]
This theorem also holds in the case where $X$ admits quasiprojective toric embeddings and the limit is taken over all such embeddings \cite{P09}. An analogous picture holds in the real setting, i.e., if $X$ is a semialgebraic set and $X^{\operatorname{an}}_r$ denotes its real analytification \cite{JSY22}, then
\[
X^{\operatorname{an}}_r 
\;\cong\; 
\lim_{\iota} \Trop_{r}(X,\iota),
\]
where $\Trop_{r}(X,\iota)$ denotes the real tropicalization of $X$ with respect to the embedding $\iota$.

\subsection*{Limit theorems via bands}

The preceding ideas come together in a natural way in the language of bands. In the non-Archimedean setting, the tropical band $\mathbb{T}$ plays the role of the tropical base object. When the tropical semifield is realized as the tropical hyperfield, one has $X^{\operatorname{an}} \cong X(\mathbb{T})$ for affine $k$-schemes $X$ (see \cite{Jun2021GeometryOfHyperfields}). Likewise, in the real setting for an affine $k$-scheme $X$ where $k$ is a real closed field, one has $X_r^{\operatorname{an}} \cong X(\mathbb{R}\mathbb{T}),$
where $\mathbb{R}\mathbb{T}$ denotes the real tropical hyperfield.




If $X=\Spec R$ is an affine $k$-scheme and $\iota$ is an affine embedding of $X$, then one can associate to $\iota$ an affine band scheme $Y_\iota$ such that $Y_\iota(\mathbb{T}) \cong \Trop(X,\iota).$
Thus the usual tropicalization map is induced by a morphism of affine band schemes after passing to $\mathbb{T}$-points. In this sense, tropicalization admits a scheme-theoretic incarnation in the category of band schemes.

A natural question is therefore whether Payne-type limit theorems admit a conceptual formulation in this framework of band schemes. More precisely, one may ask whether the inverse limit of the band schemes $Y_\iota$, taken over the natural morphisms induced by affine embeddings, recovers $X$, and whether passing to $\mathbb{T}$-points recovers the classical limit theorem. The aim of this note is to show that this is indeed the case in the affine setting. We also explain how the same formalism recovers the real tropical limit theorem of \cite{JSY22}. The more general case of toric embeddings and non-affine schemes will be treated in subsequent work.



This note is organized as follows. We first give a short introductory survey of the theory of bands, including examples (Section~\ref{Section: Bands}) and the construction of band schemes (Section~\ref{Section: Band schemes}). In Section~\ref{Section: Limitsoftrops}, we show how Payne's limit theorem for tropicalizations of affine varieties \cite{P09} can be formulated and generalized within this framework, and how the same formalism recovers the real tropical limit theorems of \cite{JSY22}.

\subsection*{Acknowledgements}

This project began during a visit of Oliver Lorscheid and Alejandro Mart\'inez M\'endez to Frankfurt, where we discussed recent developments on bands and limit theorems for tropicalizations over various hyperfields. Part of this work was carried out during the Workshop \emph{Geometry over Semirings}, held at the Universitat Aut\`onoma de Barcelona in July 2025. We are grateful to the organizers, Marc Masdeu and Joaquim Ro\'e, for creating a friendly and inspiring atmosphere to work. We thank Oliver Lorscheid, Martin Ulirsch, Kevin Kühn and Jeffrey Giansiracusa for helpful conversations, suggestions, and comments. 
Special thanks are due to Oliver for answering many detailed questions,
several of which will be taken up in a follow-up paper, and to Oliver and Martin for comments on an earlier draft.

This project has received funding from the Deutsche Forschungsgemeinschaft (DFG, German Research Foundation) TRR 326 \emph{Geometry and Arithmetic of Uniformized Structures}, project number 444845124, and TRR 358 \emph{Integral Structures in Geometry and Representation Theory}, project number 491392403, as well as from the Deutsche Forschungsgemeinschaft (DFG, German Research Foundation) Sachbeihilfe \emph{From Riemann surfaces to tropical curves (and back again)}, project number 456557832 and the DFG Sachbeihilfe \emph{Rethinking tropical linear algebra: Buildings, bimatroids, and applications}, project number 539867663, within the
SPP 2458 \emph{Combinatorial Synergies} and from the \emph{Secretaría de Ciencia, Humanidades, Tecnología e Innovación} (SECIHTI, previously CONACyT).

\section{Bands}\label{Section: Bands}
Bands and band schemes were first developed in \cite{https://doi.org/10.1112/jlms.70125}. In this section, we briefly review the basic theory, with a focus on the connection to tropical geometry.

\subsection{Bands}

\begin{df}
    A \textit{pointed monoid} is a commutative monoid $B$ that has an absorbing element $0\in B$, i.e.,  $0\cdot b=0$ for all $b\in B$. A \textit{morphism of pointed monoids} is a morphism of monoids $f\colon B\rightarrow C$ such that $f(0)=0$. 
\end{df}

\begin{df}
    The \textit{free semiring} over a pointed monoid $B$ is the commutative semiring $\mathbb{N}[B]$ of finite formal sums of elements of $B$, with the empty sum as the neutral element. Let $\sim$ be the equivalence relation given by identifying $\sum a_i\sim 0_B+\sum a_i$ for all $\sum a_i\in \mathbb{N}[B]$. The quotient semiring $B^{+}\coloneqq\mathbb{N}[B]/\sim$ is called the \textit{ambient semiring} of $B$. 
\end{df}

\begin{df}\label{df: Ideals}
    An \emph{ideal} of a pointed monoid $B$ is a $B$-submodule $I$ of $B$, that is, a proper subset $I\subseteq B$ such that $BI=I$. Likewise, an ideal of an ambient semiring $J\subseteq B^{+}$ is a $B^{+}$-submodule: a proper additive sub-semigroup that absorbs the multiplication of $B^{+}$, i.e., such that $B^{+}J= J$ and $J+J=J$.
\end{df}

\begin{df}
    A \textit{band} is a pair $B=(B,N_B)$ consisting of a pointed monoid $B$ called the \emph{underlying monoid} together with an ideal $N_B\subseteq B^+$, called the \emph{null set}, satisfying the following property: for every $a\in B$, there exists a unique $b\in B$ such that $a+b\in N_B$, which we denote by $-a\coloneqq b$. An element $\sum b_i$ of the null set $N_B$ is called a \emph{null-sum}. A band $B$ is called a \emph{tract} if, for every non-zero $b\in B$, there exists $c\in B$ such that $bc=1$. 
 \end{df}

\begin{rem}\label{rem: Tracts}
The original definition of a tract in \cite{BakerBowler2019Matroids} is a pair $(G,N_G)$ consisting of a commutative group $G$ along with a subset $N_G\subseteq \mathbb{N}[B]$ such that
\begin{itemize}
    \item $0_{\mathbb{N}[B]}\in N_G$
    \item  $1_G\notin N_G$ and there exists a unique $\epsilon\in G$ such that $1+\epsilon\in N_G$
    \item  $GN_G=N_G$ 

\end{itemize}
 The definition of tract in this text is slightly stronger than the original notion as it requires the null set to be an ideal in the ambient semiring, the cases which interest us fulfill precisely this condition. We shall consider tracts in this stronger sense.  
\end{rem}
 
 \begin{df}
     A morphism of bands $f:B\rightarrow C$ is a morphism of pointed monoids $f:B\rightarrow C$ such that $\sum f(b_i)\in N_C$ for every $\sum b_i\in N_B$. We denote by $\operatorname{Bands}$ the category of bands.
 \end{df}

    The category of bands is complete and cocomplete \cite[Cor. 1.43]{https://doi.org/10.1112/jlms.70125}. 

\subsection{Examples of bands and bandification}

\subsubsection{Bandification}
There are fully faithful embeddings  into the category $\operatorname{Bands}$ from  many algebraic categories \cite{https://doi.org/10.1112/jlms.70125}. The image of this embedding  will be called \textit{bandification} and the cases for rings and hyperrings will be described below.



\begin{ex}[Bandification of a ring]
Every ring $R$ can be naturally viewed as a band: the underlying monoid is $(R,\cdot)$ and the null set is
\begin{equation*}
N_R=\Bigl\{\sum a_i\in R^+\: :\: \sum a_{i}=0\in R\Bigr\}.
\end{equation*}
The bandification of $R$ is denoted as $\underline{R}:=(R,N_R)$ .
A ring homomorphism $R\to S$ induces a band morphism $\underline{R}\to \underline{S}$ and in this way we obtain a fully faithful functor $\Rings\to \Bands$. 
\end{ex}


\begin{ex}[Bandification of a hyperring] More generally, given a hyperring $(H,\cdot,\boxplus)$, its underlying monoid is $(H,\cdot)$ and the null set is

\begin{equation*}
N_{H}=\Bigl\{ \sum a_{i} \in H^{+} \:  : \:  0\in \boxplus a_{i}\subseteq H \Bigr\}
\end{equation*}

The bandification of $H$ is denoted as $\underline{H}:=(H,N_H)$. As in the previous case, a hyperring morphism induces a band morphism and we obtain a fully faithful functor $\operatorname{Hyperrings}\to\operatorname{Bands}$. 
\end{ex}
\subsubsection{Examples of bands}

\begin{ex}[The initial and terminal bands] The initial object in the category $\Bands$ is the \emph{regular partial field} which consists of the underlying monoid $\mathbb{F}_1^{\pm}:=\{0,1,-1\}$ (with the obvious multiplication) together with the null set
\begin{equation*}
    N_{\mathbb{F}_1^{\pm}}=\bigl\{n\cdot1+n \cdot (-1): n\in \mathbb{N}\bigr\}
\end{equation*}
The terminal object in $\Bands$ is the trivial band $\{\bm{0}\}$.\\

The Krasner band $\mathbb{K}:=\{0,1\}$ has the obvious underlying monoid $(\mathbb{K},\cdot)$ and null set
\begin{equation*}
N_{\mathbb{K}}=\bigl\{n \cdot 1 :\: n\in \mathbb{N}\setminus\{1\}\bigr\}
\end{equation*}
The Krasner band is the bandification of the Krasner hyperfield and is the terminal object in the category of tracts.
\end{ex}

\begin{ex}[The sign band]\label{thm:examplesign}
The \emph{sign band} $\mathbb{S}=\{0,1,-1\}$ has underlying monoid $(\mathbb{S},\cdot)$ with the obvious multiplication and null set
\begin{equation*}
N_{\mathbb{S}}=\bigl\{ n\cdot 1+m \cdot (-1): nm\neq 0\: \textup{or}\:n=m=0, \:\:\:n\in \mathbb{N}\bigr\}
\end{equation*}
This is the bandification of the sign hyperfield.

An ordering on a field $K$ is exactly the same as a band morphism $\underline{K} \to \mathbb{S}$ (see \cite[Example 2.23]{BakerBowler2019Matroids}).
\end{ex}

\begin{ex}[The tropical band]
\label{thm:exampletropical}
The \emph{tropical band} $\mathbb{T}$ has underlying monoid $(\R_{\geq 0},\cdot)$ and null set
\begin{equation*}
N_{\mathbb{T}}:=\{0\}\cup\Bigl\{\sum a_i\: : \:\operatorname{max}\{a_{i}\}\:\textup{occurs at least twice}\Bigr\}.    
\end{equation*}
Note that this is the bandification of the realization of the tropical semiring as a hyperfield (see, for example, \cite{viro2010hyperfieldstropicalgeometryi}) and that a valuation on a field $v:k\to \mathbb{R}_{\geq 0}$ is the same as a band morphism $\underline{v}:\underline{k} \to \T$. 
\end{ex}

\begin{ex}[The real-tropical band]\label{thm:examplerealtropical}
The \emph{real-tropical band} $\mathbb{R}\mathbb{T}$ has underlying monoid $(\R ,\cdot)$ and null set
\begin{equation*}
N_{\mathbb{R}\mathbb{T}}=\{0\}\cup\Bigl\{\sum a_{i}\: :\:  \max\{|a_i|\} \textup{ occurs at least twice and with opposite signs}\Bigr\}
\end{equation*}
This is the bandification of the real tropical hyperfield and can be seen as a \emph{tropical extension} of the sign hyperfield $\mathbb{S}$ by the ordered monoid (with reverse order) $(\mathbb{R}_{\geq 0},\times)$ \cite{smith2024matroidstropicalextensionstracts}.

Let $k$ be an ordered field with a compatible valuation, that is a valuation $\lvert \cdot\rvert:k\to\mathbb{R}$ such that for all $0 \le a \le b\in k$ then $ \lvert a \rvert_k \le \lvert b \rvert_k$. Then the signed absolute value $v^\pm:k\to \mathbb{R},\:\: x \mapsto \sgn(x) \vert x\vert_k$ is a morphism of bands $\underline{v^\pm}:\underline{k} \to \mathbb{R}\mathbb{T}$(see \cite[\protect\textsection~5]{JSY22})
\end{ex}

\begin{ex}[The triangle band] \label{ex:triangle}
The \emph{triangle band} \(\triangle\) has
\((\mathbb{R}_{\geq 0},\cdot)\) as its underlying monoid and null set
\[
N_{\triangle}
=
\{0\}
\cup
\left\{
\sum_i a_i \;:\;
\max_i\{a_i\}\leq \sum_{j\neq i} a_j
\right\}.
\]

Thus \(\sum_i a_i\in N_\triangle\) precisely when the numbers \(a_i\)
are the side lengths of a possibly degenerate $n$-gon in the Euclidean plane. \\

More generally, for \(q>0\), the \(q\)-triangle band \(\triangle_q\) has
underlying monoid \((\mathbb{R}_{\geq 0},\cdot)\) and null set
\[
N_{\triangle_q}
=
\{0\}
\cup
\left\{
\sum_i a_i \;:\;
\max_i\{a_i^{1/q}\}\leq \sum_{j\neq i} a_j^{1/q}
\right\}.
\]
The tropical band appears as the limiting
object \(\triangle_0\). Recently, a
surprising connection between Lorentzian polynomials in the sense of
Br\"and\'en and Huh and realization spaces over triangular bands has led
to new insights into the topology of these spaces (see
\cite{baker2025lorentzianpolynomialsmatroidstriangular}).
\end{ex}

\section{Band schemes}\label{Section: Band schemes}
In order to have a reasonable geometric theory, we want to associate to every band a topological space with a structure sheaf of bands, in analogy with classical affine schemes. To define the points of an \emph{affine band scheme}, we first discuss prime ideals of bands. For further details, we refer to \cite{https://doi.org/10.1112/jlms.70125}.

\begin{df}
    A multiplicative ideal, or \emph{$m$-ideal}, of a band B is a proper subset $I\subseteq B$ such that $B\cdot I\subseteq I$, that is, it is an ideal of the underlying monoid.  It is called \emph{prime} if $S=B\setminus I$ is a multiplicative set, (that is,  $1\in S$ and if $a,b\in S$ then $ab\in S$). 
 \end{df}

\begin{rem}
    This construction coincides with the notion of an ideal in $B$ being a  $B$-submodule in the category $\Bands$ by identifying the monoid ideal $I\subseteq B$ with the pair $(I,N_I=I^+\cap N_B)$, which fulfills both parts of Definition \ref{df: Ideals}.
\end{rem}
\begin{rem}
    It is easy to note that, given a band $B$, there is a unique maximal ideal $\mathfrak{m}=B\setminus B^\times$ where $B^\times=\{b\in B\: :\:  bc=1\:\textup{for some }\: c\in B\}.$
\end{rem}

\begin{ex} 
    Given a tract $k$, the unique prime $m$-ideal of $k$ is $\{0\}$ (this is the case when $k$ is the bandification of a field or hyperfield) and the prime $m$-ideals of $k[x_{i}]_{i\in I}$ (Definition \ref{def:free}) are of the form $\langle x_j \rangle_{j\in J}$ for some subset $J\subseteq I$.  
\end{ex}

\begin{ex}\label{ex:Primes of Z}
    The prime $m$-ideals of $\underline{\mathbb{Z}}$ are unions of prime ideals of $\mathbb{Z}$, and are therefore in bijection with sets of prime numbers. The unique maximal ideal is $\mathfrak{m}=\mathbb{Z}\setminus \{-1,1\}$.
\end{ex}

\begin{df}
 Let $S\subseteq B$ be a multiplicatively closed subset with $1\in S$. The \emph{localization} of $B$ in $S$ is the band $S^{-1}B$ with underlying monoid $(B\times S)/\sim$ (where $(a,s)\sim (a^{\prime},s^{\prime})$ if and only if there exists $t\in S$ such that $tsa'=ts'a$) and null set
 \begin{equation*}
     N_{S^{-1}B}= \Bigl\langle \sum \frac{a_i}{1} : a_i\in N_B\Bigr\rangle 
 \end{equation*}
 The localization map $\iota_S:B\hookrightarrow S^{-1}B$ is defined by $a\mapsto \frac{a}{1}$\\

 If $S=\{ 1,h,h^2,h^3,...\}$ for some $h\in B$ then we write $B[h^{-1}]:=S^{-1}B$ and if $S=B\setminus\mathfrak{p}$ for some prime $m$-ideal $\mathfrak{p}$ we write $B_{\mathfrak{p}}:=S^{-1}B$
\end{df}

 \begin{df}
The \emph{spectrum} of a band $B$ is defined as
\begin{equation*}
    \operatorname{Spec}(B)\coloneqq\{\mathfrak{p}\subset B\:\colon\:\mathfrak{p}\:\textup{prime $m$-ideal}\}
    \end{equation*}
    with the topology generated by \emph{principal open subsets}
    \begin{equation*}
     U_h=\{\mathfrak{p}\in \operatorname{Spec}B\::\: h\notin \mathfrak{p}\},
    \end{equation*}
    where $h\in B$. We define a structure sheaf $\mathcal{O}_X$  by setting
\[
\mathcal{O}_X(U_h) = B[h^{-1}]
\]
for all $h\in B$, which has stalks
\[
\mathcal{O}_{X,\mathfrak{p}} = B_\mathfrak{p}.
\]
\end{df} 



Note that \(\mathcal{O}_X\) is uniquely determined by its values on principal opens since
the \(U_h\) form a basis for the topology of \(\operatorname{Spec} B\). The proof of the
existence of \(\mathcal{O}_X\) is easier than in usual algebraic geometry since every
non-empty spectrum \(\operatorname{Spec} B\) contains a unique closed point, which is the
maximal ideal $\mathfrak{m}_B = B \setminus B^\times$ of \(B\). In other words, $\Spec(B)$ is non-empty for any non-trivial band $B$.  

\begin{df}
    An \emph{affine band scheme} is a topological space $X$ with a sheaf of bands $\mathcal{O}_X$ which is isomorphic to the spectrum of a band $(\Spec(B),\cO_{\Spec(B)})$
\end{df}


A band morphism \(f : B \to C\) induces a morphism
\[
\varphi = f^\ast : \operatorname{Spec} C \to \operatorname{Spec} B
\]
with \(\varphi(\mathfrak{q}) = f^{-1}(\mathfrak{q})\), and the sheaf morphism \(\varphi^\#\) is determined on
principal opens by
\[
\varphi^\#(U_h) : B[h^{-1}] \longrightarrow C[f(h)^{-1}], \qquad
\frac{a}{h^i} \longmapsto \frac{f(a)}{f(h)^i}.
\]
This yields a category $\BAff$ of affine band schemes.\\

\noindent Therefore, we obtain a  functor 
\[
 \begin{array}{cccl}
  \Spec: & \Bands & \longrightarrow & \BAff \\
         & B      & \longmapsto     & (X=\Spec(B),\cO_X)
 \end{array}
\]

\noindent which is fully faithful and, along with the \emph{global sections functor} $\Gamma:\BAff \to\Bands$ defines an antiequivalence of categories (see \cite{https://doi.org/10.1112/jlms.70125}). 



\begin{df}
    A \emph{band scheme} is a topological space $X$ with a sheaf of bands $\mathcal{O}_X$ such that every point has an affine open neighbourhood. Given a tract $k$, a $k$-band scheme is a band scheme $X$ along with a morphism $X\to\Spec(k)$. 
\end{df}

\begin{rem}
Contrary to intuition, $m$-ideals which are also additively closed on the null set, known in literature as $k$-ideals, are not the `correct' notion of ideals for a geometric theory: as stated in \cite[Theorem 2.2]{https://doi.org/10.1112/jlms.70125}, the (Grothendieck) topology with coverings generated by principal open immersions coincides with the topology given by prime $m$-ideals and thus, these correspond to the underlying points of a band scheme.
This means that the underlying topological space of the band spectrum of a ring has more points than the (classical) spectrum of the ring (for example, in \ref{ex:Primes of Z}, $\mathfrak{m}$ is not an ideal of $\mathbb{Z}$) and this contrasts to the situation for semiring schemes, which have the same underlying space as the usual spectrum in the case of a ring (see \cite{gualdi2026schemetheorycommutativesemirings}).

\end{rem}

\begin{df}\label{def:free}
Let $B$ be a band. The free $B$-band over an index set $I$ is defined as follows. Let $B[x_i]\coloneqq B[x_{i}: i\in I]$ be the monoid
\begin{equation*}
    B[x_{i}]=\bigl\{a x_{1}^{m_{1}}\dots x_{n}^{m_{n}}\::\:a\in B,\:\{m_{1},\dots,m_{n}\}\subseteq I\bigr\}
\end{equation*}
(with the obvious identification $0b\sim 0$ for all $b\in B[x_{i}]$) and the null set  is 
\begin{equation*}
    N_{B[x_i]}=\Bigl\langle \sum b_i\in B[x_i]^+ : \sum b_i\in N_B \Bigr\rangle.
\end{equation*}
\end{df}

\begin{df} \label{def:monomialbandmodel}
Let $R$ be a finitely generated $k$-algebra and let $a_1, \dots, a_n\in R$ be a set of generators for the $k$-algebra $R$. Note that this choice is equivalent to the choice of a closed embedding $\iota: X=\Spec(R) \hookrightarrow \mathbb{A}^n_k$.
The \emph{monomial band model} of $R$ with respect to $\iota$ is the
$\underline{k}$-subband $R_\iota\subseteq \underline{R}$ whose underlying
pointed monoid is the multiplicative submonoid of $R$ generated by  $a_1,\dots,a_n$, namely
$$R_\iota = \{c a^{e_1}_1\dots a^{e_n}_n : c\in k,\: e_i \in \mathbb{N} \}$$ and whose null set is induced from $\underline{R}$:
\[
N_{R_\iota}=N_{\underline{R}}\cap R_\iota^+.
\]
We call the associated  \(\underline{k}\)-band scheme  $Y_\iota = \Spec(R_\iota)$ the \emph{band scheme model of $R$ with respect to $\iota$}.
\end{df}

\begin{rem}
    Let $k$ be a field and let $\underline{k}$ be its bandification. Note that the bandification $\underline{k[x_{i}]_{i\in I}}$ of the polynomial ring $k[x_i]$ is not the same as $\underline{k}[x_{i}]$, the former is much larger: The underlying monoid of $\underline{k[x_{i}]}$ is the multiplicative monoid 
    $$\Bigl\{\sum_{a_J\in k} a_Jx^J\in k[x_{i}]_{i\in I}\::\:  J\subseteq I, \lvert J\rvert\leq\infty\Bigr\}$$
    and the underlying monoid of $\underline{k}[x_i]_{i\in I}$ is the multiplicative monoid of monomials as in Definition \ref{def:free}.
    
    There is, however, a canonical $k$-band morphism
    \begin{equation*}
        \underline{k}[x_{i}]\hookrightarrow\underline{k[x_{i}]}
    \end{equation*}
    and thus also a canonical morphism
    \begin{equation*}
    \operatorname{Spec}(\underline{k[x_{i}]})\to\operatorname{Spec}(\underline{k}[x_{i}]).
    \end{equation*}
In the same manner, if $R$ is a $k$-algebra as above, we have the canonical inclusion $R_\iota\hookrightarrow \underline{R}$ and the corresponding morphism of band schemes $\Spec\underline{R}\to Y_\iota$.
\end{rem}

\begin{df}
A \emph{topological band} is a band \(B\) together with a topology such that the multiplication map is continuous. Given a band $A$ and a topological band $B$, we define the \emph{affine topology} on $B(A):=\operatorname{Hom}(A,B)$ as the compact open topology with respect to the discrete topology on $A$. \noindent Equivalently, it is the initial topology for the evaluation maps (i.e. such that the maps $ev_a(f):\operatorname{Hom}(A,B)\to B$, $ev_a(f)=f(a)$  are continuous for all $a\in A$).\\

Given an arbitrary band scheme $X$ and a band $B$, the \emph{fine topology} on $X(B)=\Hom(\Spec(B),X)$ is the finest topology such that for every affine band scheme $Y$ and morphism $\alpha:Y\to X$ the map $\alpha_B:Y(B)\to X(B)$ is continuous with respect to the affine topology on $Y(B)$.
\end{df}

\begin{rem}
Let $k$ be a topological field. The \emph{strong topology} is the unique topology  on the $k$-rational points of $k$-schemes of finite type such that for all such schemes $X,Y$ and $Z$ the following properties hold:
\begin{itemize}
    \item $\mathbb{A}^1(k)\cong k$
    \item $(X\times_Z Y)(k)\cong X(k)\times_{Z(k)} Y(k)$
    \item An open (closed) immersion $Y\hookrightarrow X$ yields an open (closed) topological embedding $Y(k)\hookrightarrow X(k)$
\end{itemize}
For more details we refer to \cite[Section 1.10]{Mumford1999}.

    The fine topology, recovers the strong  topology  whenever $X$ is a $k$-scheme of finite type and $R$ is a topological $k$-algebra \cite[Thm. B]{LORSCHEID2016193}. 
\end{rem}

\begin{ex}\label{ex:topologicalbands}
The tropical band $\mathbb T = \mathbb R_{\ge 0}$,
endowed with the order topology \cite[Ex. 3.7]{https://doi.org/10.1112/jlms.70125}, and the signed tropical hyperfield $\R\T = \R$ endowed with the usual topology are topological bands. 
\end{ex}




\section{Limits of tropicalizations}\label{Section: Limitsoftrops}

In this section we explain how, for an affine scheme $X$ over a non-Archimedean valued field, both analytification and tropicalization can be expressed in terms of $\T$-points of affine band schemes. We then prove that the affine band scheme associated with $X$ is the inverse limit of the band schemes attached to its affine embeddings, and deduce the corresponding limit theorem on $B$-points for any topological band $B$. \\





Let $X=\operatorname{Spec}(R)$ be an affine $k$-scheme, where $R$ is a $k$-algebra. We start with the band scheme interpretation of the analytification of $X$, since such an interpretation is more direct. Recall from Example \ref{thm:exampletropical} that a non-Archimedean valuation $v$ on a field $k$ is the same as a morphism $v\colon k\to\mathbb{T}$, where $k$ and $\mathbb{T}$ are viewed as bands. To avoid confusion, when not explicitly stated, we write $\underline{k}$ to remind the reader that we are viewing $k$ as a band (similarly to algebras, etc).
\newline

If $X = \operatorname{Spec} R$ is an affine $k$-scheme of finite type,
considered as an affine band scheme over the band $k$, then
\[
  X(\mathbb{T}) = \operatorname{Hom}_{\underline{k}} (\underline{R},\mathbb{T})
\]
is homeomorphic to the Berkovich
analytification $X^{\operatorname{an}}$ of $X$ (see \cite[Thm. 3.5]{L22}, \cite[Exa. 3.8]{https://doi.org/10.1112/jlms.70125}).
\newline



Given a $k$-algebra $R$ with an associated embedding $\iota$, we have the canonical  morphism of affine $\underline{k}$-band schemes $\underline{X}=\Spec(\underline{R}) \longrightarrow Y_\iota.$ where $Y_\iota$ is the band scheme model of $R$ with respect to $\iota$ as in Definition \ref{def:monomialbandmodel}.
Passing to \(\mathbb{T}\)-points recovers the usual tropicalization map
\[
X^{\operatorname{an}} \cong X(\mathbb{T}) \longrightarrow Y_\iota(\mathbb{T}) \cong \Trop(X,\iota).
\]
where the latter homeomorphism naturally identifies \(Y_\iota(\mathbb{T})\) with the tropicalization $\Trop(X,\iota)$ of $X$ with respect to the closed embedding $\iota: X \hookrightarrow \A^n_k$ (see \cite[Exa. 3.8]{https://doi.org/10.1112/jlms.70125}).
Thus tropicalization appears as the map on \(\mathbb{T}\)-points induced by a natural morphism of affine band schemes.

In the real setting, if $X$ is an affine variety over a real closed field $k$, by an analogous argument \(Y_\iota(\mathbb{R}\mathbb{T})\) is naturally homeomorphic to the real tropicalization $\Trop_r(X,\iota)$ of $X$.
 
\begin{rem}\label{rem: Bend relations}
The Giansiracusa bend relations \cite{GiansiracusaGiansiracusa2016Equations} can be recovered from band theory. Given a ring $R=k[A]/I$ over non archimedean valued field $v:k\to\R_{\geq0}$ (and an implicit affine scheme embedding $\iota$) the \emph{bend relations} are $v(c_a)a+\sum v(c_{b_i})b_i\sim v(c_{b_i})b_i$, where $c_a,c_{b_i}\in k$, $a,b_i\in R$ and $c_aa-\sum c_{b_i}b_i\in I$. They generate a congruence relation $\bend(I)$ and therefore an idempotent semiring $\Bend(R):=\T[A]/\bend(I)$ (where $\T$ is interpreted here as the classical tropical semiring).
We can associate a band to $\Bend(R)$ in a natural way: $\underline{\Bend} (R)$ consists of monomials in $\T[A]$ and the null set is 
\[
N_{\underline{\Bend} (R)}:=\Bigl\langle v(r)a+\sum v(s_i)b_i\::\:a\otimes 1 + \sum (b_i\otimes 1)\in N_{\underline{R}\otimes\T} \Bigr\rangle
\]
Then, taking $X=\Spec(R)$, we have that  $\Trop(X,\iota)\cong \underline{\Bend}(R)(\T)$ \cite{GiansiracusaGiansiracusa2016Equations}.
\end{rem}
 
Consider now two closed embeddings $  \iota\colon X\hookrightarrow \mathbb{A}^n_k,  \jmath\colon X\hookrightarrow \mathbb{A}^m_k,$
and a toric morphism $  \phi : \A^n_k \to \A^m_k$
such that $\jmath = \phi \circ \iota$. Then each coordinate of $\jmath$ is a monomial in the coordinates of $\iota$, and this induces an inclusion of $\underline{k}$-bands $ R_\jmath \subset R_\iota.$
Equivalently, we obtain a morphism of affine $\underline{k}$-band schemes
\[
  \varphi^*_{\iota\jmath}:Y_\iota \longrightarrow Y_\jmath.
\]



\begin{thm}
\label{thm:mainthm}
Let $R$ be a finitely generated $k$-algebra and let $X = \Spec(\underline{R})$ be the associated affine $\underline{k}$-band scheme. Then there is a natural isomorphism of band schemes 
       $$X \cong \lim_\iota Y_\iota$$
       where the limit runs over all closed embeddings $\iota: X \hookrightarrow \A^{n(\iota)}_k$ and toric morphisms between them as above. \\
Moreover, for a $k$-band (resp. topological $k$-band) $B$, it induces a bijection (resp. homeomorphism) 
       $$X(B) \cong \lim_\iota Y_\iota(B).$$
\end{thm}

\begin{proof}
By construction, the subbands \(R_\iota\subseteq \underline R\) are
cofinal among finite collections of elements of \(R\): given
\(r_1,\dots,r_s\in R\), we may extend them to a finite set of
\(k\)-algebra generators of \(R\), hence to an affine embedding \(\iota\),
so that \(r_1,\dots,r_s\in R_\iota\). Therefore every element of
\(\underline R\), and every finite null-sum of \(\underline R\), already
lies in some \(R_\iota\). Since \(N_{R_\iota}=N_{\underline R}\cap
R_\iota^+\), the canonical map
\[
    \colim_\iota R_\iota \longrightarrow \underline R
\]
is an isomorphism of bands.

Since $\Spec$ is a contravariant equivalence between bands and affine band schemes, 
it follows that
\[
X = \Spec(\underline{R})
\;\cong\;
\lim_{\iota}\Spec(R_\iota)
=
\lim_{\iota}Y_\iota .
\]

Now let $B$ be a \(\underline{k}\)-band. We obtain
\[
X(B)
=
\Hom(\Spec B,X)
\cong
\Hom(\Spec B,\lim_{\iota}Y_\iota)
\cong
\lim_{\iota}\Hom(\Spec B,Y_\iota)
=
\lim_{\iota}Y_\iota(B),
\]
which proves the claimed bijection.


If $B$ is a topological $\underline{k}$-band, then this bijection is a homeomorphism by a specialization of \cite[Thm. 6.4]{Lorscheid2023UnifyingTropicalization}
\end{proof}

This recovers Payne's limit theorem \cite{P09} by taking $\T$-points:
\begin{cor}
    In the case where $B=\mathbb{T}$, then $X^{an}\cong X(\mathbb{T})\cong \lim_{\iota}Y_\iota(\mathbb{T})
\;\cong\;
\lim_{\iota}\Trop(X,\iota).$
\end{cor}

Taking $\mathbb{R}\mathbb{T}$-points recovers the real limit theorem of \cite{JSY22}:
\begin{cor}
Let $k$ be an ordered field endowed with a compatible non-Archimedean valuation, and let $X=\Spec R$ be an affine $k$-scheme of finite type. Then
\[
X_r^{\operatorname{an}}
\;\cong\;
X(\mathbb{R}\mathbb{T})
\;\cong\;
\lim_{\iota}Y_\iota(\mathbb{R}\mathbb{T}),
\]
and $Y_\iota(\mathbb{R}\mathbb{T})$ identifies with the real tropicalization of $X$ with respect to $\iota$.
\end{cor}

\begin{rem}
Theorem~\ref{thm:mainthm} can be applied to other topological bands as
well. For instance, applying it to the triangle bands \(\Delta_q\)
(cf. Example~\ref{ex:triangle}) identifies \(X(\Delta_q)\) with the
corresponding inverse limit of the spaces \(Y_\iota(\Delta_q)\). Since
\(\Delta_q\) is the bandification of the $q$-triangular hyperfield,
these spaces are closely related to the \(q\)-amoebas appearing in
Maslov dequantization (see
\cite[Section~3.5]{baker2025lorentzianpolynomialsmatroidstriangular}).

However, this relationship should not be interpreted as a literal identification
with the classical \(q\)-amoeba. For a fixed embedding \(\iota\), the
classical \(q\)-amoeba maps into the corresponding \(\Delta_q\)-valued
space considered in
\cite[Prop. 3.36]{baker2025lorentzianpolynomialsmatroidstriangular}, and this space may be
strictly larger; see also
\cite[Rem. 3.37]{baker2025lorentzianpolynomialsmatroidstriangular}. An interesting question to explore in a follow up work is to
understand more precisely to what extent \(X(\Delta_q)\) provides an
amoeba-theoretic analogue of the Berkovich analytification.
\end{rem}

\nocite{kuronya2025tropicalization}

 \printbibliography

\end{document}